\documentclass[preprint,12pt]{amsart}

\usepackage{xcolor}
\usepackage[margin=1.3in]{geometry}
\usepackage{pdftexcmds}
\usepackage{amsmath,calligra}
\usepackage{amsfonts}
\usepackage{amssymb}
\usepackage[all]{xy}
\usepackage{longtable}
\usepackage{amsthm}
\usepackage{booktabs}
\usepackage{caption}
\usepackage{mathrsfs}
\usepackage{times}
\usepackage{enumerate}

\usepackage[colorlinks=true,citecolor=blue, urlcolor=blue, linkcolor=blue,pagebackref]{hyperref}
\usepackage[normalem]{ulem}

\usepackage{setspace}

\AtBeginDocument{\addtocontents{toc}{\protect\setlength{\parskip}{0pt}}}

\def\bZ{\ensuremath{\mathbb{Z}}}
\def\bP{\ensuremath{\mathbb{P}}}

\def\cA{\ensuremath{\mathcal{A}}}
\def\cB{\ensuremath{\mathcal{B}}}

\def\cF{\ensuremath{\mathcal{F}}}

\def\cJ{\ensuremath{\mathcal{J}}}
\def\cM{\ensuremath{\mathcal{M}}}
\def\cO{\ensuremath{\mathcal{O}}}
\def\cP{\ensuremath{\mathcal{P}}}
\def\cR{\ensuremath{\mathcal{R}}}
\def\cT{\ensuremath{\mathcal{T}}}
\def\cY{\ensuremath{\mathcal{Y}}}

\def\tC{\ensuremath{\tilde{C}}}

\DeclareMathOperator{\Aut}{Aut}
\DeclareMathOperator{\Pic}{Pic}
\DeclareMathOperator{\Nm}{Nm}

\DeclareMathOperator{\Fix}{Fix}
\DeclareMathOperator{\id}{id}
\DeclareMathOperator{\sheafhom}{\mathscr{H}\text{\kern -3pt {\calligra\large om}}\,}
\newcommand{\set}[1]{\left\{#1\right\}}

\usepackage{aliascnt}

\theoremstyle{plain}
\newtheorem{theorem}{Theorem}[section]

\newaliascnt{corollary}{theorem}
\newtheorem{corollary}[corollary]{Corollary}
\aliascntresetthe{corollary}
\newaliascnt{lemma}{theorem}
\newtheorem{lemma}[lemma]{Lemma}
\aliascntresetthe{lemma}
\newaliascnt{proposition}{theorem}
\newtheorem{proposition}[proposition]{Proposition}
\aliascntresetthe{proposition}

\theoremstyle{definition}
\newaliascnt{definition}{theorem}

\aliascntresetthe{definition}
\newaliascnt{example}{theorem}

\aliascntresetthe{example}
\newaliascnt{rem}{theorem}
\newtheorem{rem}[rem]{Remark}
\aliascntresetthe{rem}

\makeatletter
\def\ps@pprintTitle{%
 \let\@oddhead\@empty
 \let\@evenhead\@empty
 \let\@oddfoot\@empty
 \let\@evenfoot\@oddfoot
}
\makeatother

\title{Pseudoreflections on Prym Varieties}
\author[R. Auffarth, M. Lahoz, J.C. Naranjo]{Robert Auffarth, Mart\'{\i} Lahoz, Juan Carlos Naranjo}
\address{R. Auffarth \\Departamento de Matem\'aticas, Facultad de
Ciencias, Universidad de Chile, Santiago\\Chile}
\email{rfauffar@uchile.cl}
\address{M. Lahoz \newline 1. Departament de Matem\`atiques i Inform\`atica,
Universitat de Barcelona, Gran Via de les Corts Catalanes, 585, 08007 Barcelona, Spain \newline 2. Centre de Recerca Matemàtica, Edifici C, Campus Bellaterra, 08193 Bellaterra, Spain }
\email{marti.lahoz@ub.edu}
 \address{J.C. Naranjo \newline 1. Departament de Matem\`atiques i Inform\`atica,
Universitat de Barcelona, Gran Via de les Corts Catalanes, 585, 08007 Barcelona, Spain \newline 2. Centre de Recerca Matemàtica, Edifici C, Campus Bellaterra, 08193 Bellaterra, Spain }
 \email{jcnaranjo@ub.edu}


\begin{document}

\begin{abstract}
We show that for every $g\geq 5$, the locus of Prym varieties in $\cA_{g-1}$ that possess a pseudoreflection of geometric origin is the union of three different non-empty explicit irreducible families.
This is in stark contrast to the loci of Jacobian varieties that possess a pseudoreflection of geometric origin, which is empty for any genus greater than 3.
In genus $6$, a distinguished example of Prym varieties with a pseudoreflection is given by intermediate Jacobians of cubic threefolds that possess an Eckardt point.
\end{abstract}

\thanks{R. Auffarth is partially supported by Fondecyt grant 1220997. M. Lahoz and J.C. Naranjo are partially supported by the Departament de Recerca i Universitats de la Generalitat de Catalunya (2021 SGR 00697), the Proyecto de Investigaci\'on PID2023-147642NB-100 and the Spanish State Research Agency, through the Severo Ochoa and María de Maeztu Program for Centers and Units of Excellence in R\&D (CEX2020-001084-M)}

\maketitle

\section{Introduction}

In \cite[Theorem~1.2]{ALA} it was shown that if $C$ is a smooth projective curve of genus $g\geq 1$ and $G\leq\Aut(C)$ is a finite subgroup such that $JC/G$ is smooth, then $g\leq 3$ and the pair $(C,G)$ can be described explicitly.
In other words, if a Jacobian variety has a smooth quotient by a finite group whose action \textit{comes from the curve}, then its dimension is at most $3$.
The natural question to ask then is if a similar result holds true for Prym varieties.

Consider the following: Let $C$ be a smooth projective curve of genus $g\geq 1$, let $\eta\in JC[2]$ be a non-trivial 2-torsion point, and let $G\leq\Aut(C)$ be a finite group of automorphisms such that each of its elements fixes $\eta$.
Then $\eta$ induces an \'etale double covering
\[f:\tC\to C\]
given by a fixed point free involution $\iota\in\Aut(\tC)$, and there exists an extension 
\[1\to\langle\iota\rangle\to\tilde{G}\to G\to 1\]
such that $\tilde{G}\leq\Aut(\tC)$ and $f$ is $\tilde{G}$-equivariant.
Let $P:=P(\tC,C)$ denote the Prym variety of $f$.
Since $\iota$ acts on $J\tC=f^*JC+P$ as the identity on $f^*JC$ and as $-1$ on $P$, and since the action of $\tilde{G}$ on $J\tC$ by pullback also preserves both of these factors, we have that $\iota$ is central in $\tilde{G}$.
We will say that any automorphism on $P$ that fixes the polarization and that comes from an automorphism on $C$ in this way is \textit{of geometric origin}.

We recall that if $H$ is a group that acts on a manifold $X$, then the Chevalley-Shepard-Todd Theorem states that $X/H$ is smooth if and only if for every $x\in X$, the group $\operatorname{Stab}_H(x)$ is generated by \textit{pseudoreflections}, where a pseudoreflection (at $x$) is an element of the group whose fixed locus passing through $x$ is of pure codimension 1.

We would like to understand the following question:\\

\textbf{Question:} When does there exist $\sigma\in G$ that can be lifted to $P$ as a pseudoreflection?\\

If $G$ contains such an element, then we would have in particular that $P/\sigma$ is smooth.
The main theorem of this article states that, unlike the case of Jacobians, there are Prym varieties of arbitrarily large dimension with a smooth quotient.
Let $\cR_g$ denote the moduli space of \'etale double covers of curves of genus $g$, and let 
\[\cP_g:\cR_g\to \cA_{g-1}\]
denote the Prym map.
Our main theorem states the following:

\begin{theorem}
For every $g\ge5$, the locus of Prym varieties in $\cA_{g-1}$ that possess a pseudoreflection of geometric origin is the union of three different explicit irreducible families, one of dimension $2g-4$ and the other two of dimension $2g-3$.
\end{theorem}

These families are described in Subsection~\ref{g(F)=0} and Theorems~\ref{RBg2} and \ref{RFg3}, and are shown to be different in Subsection~\ref{different}.
An interesting example is given in the case of dimension $5$ in Subsection~\ref{Eckardt}, where we show that Prym varieties with a pseudoreflection appear naturally as intermediate Jacobians of cubic threefolds with Eckardt points (that is, points such that there are infinitely many lines in the cubic threefold passing through them).
We finish the article by giving a restriction on the possible groups of geometric origin that can act on Prym varieties with smooth quotient; there are at most $5$, and this analysis is done in Section~\ref{listofgroups}.

\subsubsection*{Acknowledgements}
The first author was partially supported by the ANID - Fondecyt grant 1220997 as well as the Math-AmSud grant GV-BCEF 220010.
The second and the third authors were partially supported by the Departament de Recerca i Universitats de la Generalitat de Catalunya (2021 SGR 00697), the Proyecto de Investigaci\'on PID2023-147642NB-100 and the Spanish State Research Agency, through the Severo Ochoa and María de Maeztu Program for Centers and Units of Excellence in R\&D (CEX2020-001084-M).

\section{Preliminaries}
In this section we will review a few basic preliminaries needed for the sequel.

\subsection{Prym varieties}\label{subsec:PrymVar} Let $C$ be a smooth projective curve of genus $g$ and let $\eta\in JC[2]$ be a non-trivial 2-torsion point; $\eta$ induces an \'etale double cover 
\[f:\tC\to C\]
associated to an involution $\iota\in\Aut(\tC)$ such that $f_*\mathcal{O}_{\tC}\cong\mathcal{O}_C\oplus\eta.$ Moreover, $\langle \eta \rangle$ is the kernel of $f^* : JC \rightarrow J\tC $; remember that in the ramified case $f^*$ is injective.
The universal property of the Jacobian variety induces the so-called norm map $\Nm_f:J\tC \rightarrow JC$ which is a surjective map of abelian varieties.
The component of the origin of the kernel of $\Nm_f$ is called the Prym variety of $f$ and is denoted either by $P(f)$ or $P(\tC,C)$.
Note that its dimension is $g-1=g(\tC)-g$.
The restriction of the polarization of $J\tC$ to $P(f)$ is twice a principal polarization.
Hence, denoting by $\cR_g$ the moduli space of isomorphism classes of double coverings $f:\tC\rightarrow C$ as above, we have a map:
\[
\cP_g:\cR_g \to \cA_{g-1}, 
\]
which is usually called the Prym map, and the (closure of the) image is called the Prym locus.
We refer to \cite{do_fibres} for the main properties of $\cP_g$.
In particular, for $g\ge 7$, this map $\cP_g$ is generically injective (but never injective).
For $g=6$ the source and the target have the same dimension and $\cP_6$ is generically finite of degree $27$.
Finally, for $g\le 5$, $\cP_g$ is dominant with positive-dimensional general fiber.
Note that all Prym varieties of dimension $\le 3$ are either Jacobians or products of Jacobians, hence we impose the restriction $\dim P(C,\eta)=g-1\ge 4$, that is $g\ge 5$.

Although we are interested in Prym varieties, our main examples are isogenous to products of abelian varieties obtained via the Prym construction applied to ramified coverings of curves.
Let us consider the moduli space: 
\[
\cR_{g,r}=\{f:D\rightarrow C \mid g(C)=g, f \text{ double covering ramified in } r \text{ points} \}/\cong,
\]
where $D,C$ are smooth irreducible curves, $g\ge 1$, and $r>0$ is an even integer.
This moduli space is also irreducible (see, for instance, \cite[Proposition~2.5]{Bu} for $g\ge 2$, a similar monodromy argument works for $g=1$).

Then, $P(f)$ is defined, as above, as the kernel of the norm map $\Nm_f$ (which now is connected), and the polarization on $JD$ restricts to a polarization of type 
\[
\delta:=(1,\ldots,1,2,\ldots,2),
\]
where $2$ appears $g$ times.
The ramified Prym map attaching $P(f)$ to $f$ is denoted by:
\[
\cP_{g,r}: \cR_{g,r} \to \cA_{g-1+\frac r2}^{\delta}.
\]
It is known that $\cP_{g,r}$ is injective for $r\ge 6$ (see \cite{Ik} and \cite{NO}).

\subsection{Smooth quotients}
Let $G$ be a finite group acting holomorphically on a complex manifold $X$.
An element $\sigma\in G$ is a \textit{pseudoreflection} at $x\in X$ if $\sigma(x)=x$, and the fixed locus of $\sigma$, in a neighborhood of $x$, is of pure codimension 1.
The Chevalley-Shephard-Todd Theorem states that $X/G$ is smooth if and only if for every $x\in X$, $\operatorname{Stab}_G(x)$ is generated by pseudoreflections at $x$.
In particular, if $G\cong\bZ/2\bZ$, then $X/G$ is smooth if and only if the fixed locus of the non-trivial element of $G$ is of pure codimension 1.

In \cite{ALA1}, a characterization was given of smooth quotients of abelian varieties by finite groups that fix the origin, and in \cite{ALA}, it was shown which principally polarized abelian varieties have a smooth quotient such that the group acting also preserves the polarization.
We will revisit this in Section~\ref{listofgroups} in more detail.
Note that if $\sigma$ is an automorphism of an abelian variety that fixes the origin, then each component of the fixed locus of $\sigma$ has the same dimension.
For this reason, it makes sense to talk about a pseudoreflection on the abelian variety, without specifying at which point.

By the results in \cite{ALA1}, we have that any pseudoreflection on an abelian variety is of order $2$, $3$, $4$ or $6$, and in the next section we will show that in the case of the existence of a pseudoreflection on a Prym variety with a geometric origin, it must have order $2$ or $4$.

\section{Prym varieties with an odd prime order pseudoreflection of geometric origin}
\label{sec:pseudOrigin}

Our main goal is to determine families of Prym varieties with a pseudoreflection.
As above $C$ is a projective smooth curve of genus $g$ and $\eta$ is a non-trivial $2$-torsion point in $JC$ which induces a double étale cover $f:\tilde{C}\to C$.
We will start by considering any automorphism $\sigma\in\Aut(C)$ such that $\sigma^*\eta\cong\eta$ and study its behavior; we will later impose the condition that $\sigma$ be a pseudoreflection.

It is well-known that any automorphism that leaves $\eta$ invariant lifts to the cover $\tC$ induced by $\eta $, and therefore $\sigma$ induces an automorphism on $\tC$ that we will denote by the same letter.
We note that this lifting is not unique, since if $\sigma$ is one such lifting, $\iota\sigma$ is another.

First we see that $J\tC=f^*JC+P$ and that $\iota$ acts on $f^*JC$ as $\id$ and on $P$ as $-\id$.
Moreover, since the action of $\sigma$ comes from $C$ and $\sigma$ preserves the polarization of $J\tC$, we have that $\sigma$ acts on $f^*JC$ and $P$ as well.
In particular, we have that $\sigma$ and $\iota$ commute on $J\tC$, and therefore also as automorphisms of $\tC$.

In this section, we will show that automorphisms of odd prime order will never provide a pseudoreflection.
Hence, we are able to reduce the problem to the cases of order $2$ and $4$, which will be the content of the next section.

\begin{lemma}\label{lem:ord2kernel}
For $g\geq 5$, let $Z(\iota)\subseteq\operatorname{Aut}(\tilde{C})$ denote the subgroup of all automorphisms of $\tilde{C}$ that commute with $\iota$. Then the kernel of the natural homomorphism $\rho:Z(\iota)\to\operatorname{Aut}(P)$ has order at most $2$.
Moreover, if $\id\neq \mu \in \ker \rho$, then $g(\tilde C /\langle\mu,\iota\rangle)=g(\tilde C/\iota \mu)\le 1$.
\end{lemma}
\begin{proof}
Let $h$ be the genus of the quotient curve  $\tilde{C}/\ker(\rho )$, then by Riemann-Hurwitz we have that: $2\tilde{g}-2=\left|\ker\rho\right|(2h-2)+r$, hence $\tilde{g}-1=\left|\ker\rho\right|(h-1)+r/2$, where $r$ is the degree of the ramification divisor.
Since $\tilde{g}=2g-1$, we obtain  $2g-2=\left|\ker\rho\right|(h-1)+r/2$.
 
On the other hand, all the elements in  $\ker\rho$ act trivially on $P$,  which implies that $P$ is a factor of the Jacobian of $\tilde{C}/\ker\rho$.
In fact, $P$ is isogenous to the  Prym  variety of the map: $\tilde{C}/\ker(\rho)\to\tilde{C}/\langle\ker\rho,\iota\rangle$.
In particular, $g-1=\dim P\le h$.
Therefore, $2g-2 \ge \left|\ker\rho\right|(g-2)+r/2\geq \left|\ker\rho\right|(g-2)$ and $\left|\ker\rho\right|\le \frac {2g-2}{g-2}=2+\frac{2}{g-2}$.
Since  $g\geq 5$, we obtain $\left|\ker\rho\right|\leq 2$.  

Taking the quotients of $\tC$ by $\iota$, $\iota\mu$ and $\mu$, we get the following commutative diagram:
\begin{equation}\label{diagram0}
 \xymatrix{
&\tC\ar[dl] \ar[d] \ar[dr] &\\
C\ar[dr] & \tC /\mu \ar[d] &\tC / \iota \mu\ar[dl] \\
& \tC /\langle \mu, \iota \rangle&}
\end{equation}
We note from the diagram that the Prym variety of the quotient $\tilde{C}/\iota\mu \to\tilde{C}/\langle\mu,\iota \rangle$ is $0$, since by comparing tangent spaces, this Prym corresponds precisely to the subspace of $T_0(J\tC)$ where $\iota\mu$ acts trivially and $\iota$ acts as $-1$. In other words, it corresponds to the space where $\iota$ and $\mu$ both act as $-1$, which by hypothesis is trivial. This implies that necessarily both curves have the same genus. Since the map is of degree $2$, by Riemann-Hurwitz we conclude that the genus of both curves must be either $0$ or $1$. 
\end{proof}
Note that the map $\rho:Z(\iota)\to \operatorname{Aut}(P)$ can indeed have a non-trivial kernel.
This was shown to us by an anonymous referee who pointed out that if $\tilde{C}$ is hyperelliptic with hyperelliptic involution $j$, then $\iota j$ is an involution of $\tC$, but it acts trivially on the Prym variety.

The previous lemma implies that when we talk about an automorphism of prime odd order $p$ on $P$ of geometric origin, then it actually comes from an automorphism of $C$ of order $p$.
Indeed, if $\sigma\in\operatorname{Aut}(\tC)$ is an automorphism such that $\rho(\sigma)$ is of order $p$ for an odd prime $p$, then the previous lemma implies that $\sigma$ is of order either $p$ or $2p$. In particular, $\sigma^2$ is of order $p$, which implies in particular that the restriction $\rho:\langle\sigma^2\rangle\to\langle\rho(\sigma)\rangle$ is an isomorphism, and so $\rho(\sigma)$ is indeed induced by an element of order $p$.

Let us now assume that $\sigma$ is any automorphism  of prime order $p>2$ on $\tilde{C}$, that commutes with $\iota$. By the previous lemma, we have that $\sigma$ acts on $P$ as an automorphism of order $p$.
In this case, we have that $\tC/\iota\sigma=C/\sigma$ since $\sigma$ is of odd order.
We therefore have the following commutative diagram:
\begin{equation}\label{diagram2}
 \xymatrix{
&\tC\ar[dl]_{\pi}\ar[dd]^{\pi_1}\ar[dr]^{\pi_2}&\\
\tC/\iota=C\ar[dr]_{\epsilon}&&\tC/\sigma\ar[dl]^{\epsilon_2}\\
&\tC/\iota\sigma&}
\end{equation}
 where $\pi_2$ and $\epsilon$ are of degree $p$, and $\pi$ and $\epsilon_2$ are of degree 2.

\begin{lemma}
The morphism $\epsilon_2$ is \'etale.
\end{lemma}
\begin{proof}
This comes from analyzing the possible ramification of the morphisms in diagram \eqref{diagram2}.
Indeed, the ramification of $\epsilon\pi$ comes solely from the fixed points of $\sigma$ acting on $C$, and therefore each ramification point of $\pi_1$ on $\tC$ has a stabilizer of order $p$, and $\iota$ just permutes these points.
In particular, if $\epsilon$ ramifies at $t$ points, $\pi_1$ ramifies at $2t$ points.
Since $\pi_1=\epsilon_2\pi_2$ and $\pi_2$ is of degree $p$, we must have that $\epsilon_2$ must be unramified and therefore \'etale.
\end{proof}

Now let us analyze the action of $\sigma$ on $P$.
Note that $\pi_2^*P(\epsilon_2)$ is exactly the abelian subvariety of $J\tC$ where $\iota$ acts as $-1$ and $\sigma$ acts as the identity.
In other words:
\[\pi_2^*P(\epsilon_2)=\Fix_P(\sigma)_0.\]
Let $\gamma $ be the codimension of $\pi_2^*P(\epsilon_2)$ in $P$.
We note that $\sigma$ is a pseudoreflection if and only if $\gamma=1$.
Since $\dim P(\epsilon_2)=g(\tC/\sigma)-g(C/\sigma)$, we have that
\[\gamma +g(\tC/\sigma)-g(C/\sigma)=g-1.\]
By the Riemann Hurwitz formula for $\epsilon_2$, we get that $g(\tC/\sigma)-g(C/\sigma)=g(C/\sigma)-1$, and so
\[g(C/\sigma)=g-\gamma .\]
Now if we plug in Riemann Hurwitz for $\epsilon$, we get that if $\sigma$ acts on $C$ with $t$ fixed points, then
\[
2g-2=p(2g(C/\sigma)-2)+t(p-1)=p(2g-2\gamma -2)+t(p-1),
\]
and therefore:
\begin{equation}\label{genus}
2g(p-1)=2(p-1)-t(p-1)+2p\gamma .
\end{equation}
By equation \eqref{genus}, we necessarily have that $p-1$ divides $2p\gamma $, and therefore $p-1$ must divide $2\gamma$.
From this, we easily obtain the following result:

\begin{proposition}\label{noPseudoHigherOrder} For $g\ge 3$,
any pseudoreflection on $P$ of geometric origin has
order $2$ or $4$. In particular, every Prym variety with a pseudoreflection of geometric origin admits a pseudoreflection of order $2$.
\end{proposition}
\begin{proof}
 If $\sigma$ is a pseudoreflection on $P$ of prime order $p$ for $p\geq 3$, then $\gamma =1$, and $p$ could only be $3$.
 In this case, equation \eqref{genus} becomes $2g=5-t$.
 Since $g\ge 3$, then $t\le 5-2g<0$, which is impossible. 
 
 If $\sigma$ is not of prime order, then the order of $\sigma$ is either $4$ or $6$. If it is of order $6$, then $\sigma^2$ is also a pseudoreflection but of order $3$, which is impossible by the previous argument.

 The last statement follows from noticing that if $\sigma$ is a pseudoreflection of order $4$, then $\sigma^2$ must be a pseudoreflection of order $2$.
 \end{proof}

For automorphisms of order $3$, we also get a general result on the codimension $\gamma$ that will come in handy when we look at the possible groups that can act on Prym varieties with smooth quotients (see Section~\ref{listofgroups}).

\begin{proposition}\label{order3}
If $g\geq 5$ and $\sigma\in\operatorname{Aut}(C)$ has order $3$, then the codimension $\gamma$ of $\Fix_P(\sigma)_0$ in $P$ satisfies $\gamma \geq 3$.
\end{proposition}
\begin{proof}
 Replacing $p=3$ in equation \eqref{genus}, we obtain $2g=2-t+3\gamma$.
 Then $0\le t=2-2g+3\gamma$.
 Since $g\ge 5$, we get $3\gamma \ge 2g-2 \ge 8$.
 Thus $\gamma \ge 3$.
\end{proof}

What is left to prove is that any pseudoreflection of order $2$ on $P$ of geometric origin actually comes from an involution of $C$.
The fact that orders are preserved for involutions that induce a pseudoreflection of the Prym variety follows from the following application of the Lefschetz Fixed Point Formula:

\begin{lemma}\label{lem:pseudoOford2}
Let $\sigma \in Z(\iota)$ with the following properties:
\begin{enumerate}
    \item\label{itema:ps2} as an automorphism of $\widetilde C$ it has order $4$ (i.e. $\sigma ^4=\id$ and $\sigma ^2\neq \id$), so also as an automorphism of $C$;
    \item\label{itemb:ps2} the automorphism $\sigma $ induces a pseudoreflection on $P$;
    \item\label{itemc:ps2} the automorphism $\sigma ^2$ acts trivially on $P$; and
    \item\label{itemd:ps2} the genus of $\tC/\langle \iota,\sigma^2\rangle =C/\sigma^2$ is $0$ or $1$ (hence the same holds for $C/\sigma$).
\end{enumerate}
Then $g\le 3$.
\end{lemma}
\begin{proof}
We consider the action of $\sigma $ in $T_0(J\widetilde C)=H^0(\widetilde C,\omega_{\widetilde C})^*=H^{1,0}(C)^*$.
This tangent space is the product of the tangent of $P$ and the tangent in $JC$.
By property \eqref{itema:ps2} the eigenvalues are $1, -1, \xi, \xi^{-1}$, where $\xi=e^{i\pi/4}$. Since the action comes from an action on the real vector space $H^1(\widetilde C, \mathbb R)$,
we know that there are as many eigenvalues $\xi$ as $\xi^{-1}$ in the action on $H^{1,0}(C)\oplus H^{0,1}(C)$ , we call $a$ this number.
By properties \eqref{itemb:ps2} and \eqref{itemc:ps2} the part of the diagonal matrix corresponding to $T_0(P)$ has $g-2$ times the eigenvalue $1$ and one eigenvalue $-1$.
Property \eqref{itemd:ps2} says that there is at most one eigenvalue which is either $1$  or $-1$.
Consider the case where there is an eigenvalue $\epsilon =\pm 1$.
This is the bielliptic case, the hyperelliptic case is similar.
Note that in this case $2a+1=g$.
Then the trace $T$ of the matrix of the action on $H^1(C,\mathbb C)$ is:  
    \[
    T=2(g-3 +\epsilon +\frac{g-1}2  (\xi + \xi^{-1}))=2(g-3+\epsilon+\frac{g-1}2 \sqrt 2).
    \]
Lefschetz Fixed Point Formula implies that $T\leq 2$.
All together:
    \[
    g-3+\epsilon+\frac{g-1}2 \sqrt 2\le 1.
    \]
    Then, $g (1+ \frac {\sqrt 2}2)\le 4-\epsilon +\frac {\sqrt 2}2\le 5-\epsilon .$ The worst case is $\epsilon =-1$, so $g\le 6/(1+ \frac {\sqrt 2}2)<4$.
\end{proof}

\section{Involutions}
\label{sec:involutions}

By the results in Section~\ref{sec:pseudOrigin}, 
it suffices to analyze involutions when studying moduli of Prym varieties admitting a pseudoreflection of geometric origin.
Throughout this section $\sigma$ will be an involution acting on a curve $C$ of genus $g$.

Since we are interested in cases where a lift of $\sigma$ is a pseudoreflection, we can assume that the lift we denote by $\sigma$ is the pseudoreflection.
Note moreover that if $\sigma$ is a pseudoreflection on $P$ then $\iota\sigma$ is not. Indeed, if $\sigma$ is a pseudoreflection, then its action on the tangent space of $P$ has the eigenvalue $-1$ appear with multiplicity $1$, and the eigenvalue $1$ appears with multiplicity $g-2$. Since $\iota$ acts as $-1$ on $P$, then $\iota\sigma$ has exactly the opposite multiplicities for the same eigenvalues, and therefore cannot be a pseudoreflection if $g\geq 4$.
Hence at most one lifting of $\sigma$ has the chance of having this property.

Taking the quotients of $\tC$ by $\iota$, $\iota\sigma$ and $\sigma$, we get the following commutative diagram:
\begin{equation}\label{diagram}
 \xymatrix{
&\tC\ar[dl]_{\pi}\ar[d]^{\pi_1}\ar[dr]^{\pi_2}&\\
C\ar[dr]_{\epsilon}&C_1\ar[d]^{\epsilon_1}&C_2\ar[dl]^{\epsilon_2}\\
&F&}
\end{equation}
where $\pi_1$ is the quotient by $\iota\sigma$, $\pi_2$ is the quotient by $\sigma$, and the three bottom arrows are the respective morphisms that go to the total quotient $F=\tC/\langle\iota,\sigma\rangle$.
Note that the sum of the branch divisors of $\epsilon_1 $ and $\epsilon_2$ is the branch divisor of $\epsilon $, then by Riemann Hurwitz formula we obtain:
\begin{equation} \label{sum_of_genus}
g(C_1)+g(C_2)=g+2g(F)-1.
\end{equation}

\begin{lemma}\label{involution}
The involution $\sigma$ lifts to a pseudoreflection on $P$ if and only if $g(F)=g(C_1)-1$, in which case $g(C_1)\in\{1,2,3\}$.
\end{lemma}

\begin{proof}
There is a natural map in the previous diagram, since the Prym variety of $\epsilon_1$ (respectively $\epsilon_2 $) is mapped to $P$ by $\pi_1^*$ (respectively $\pi_2^*$):
 \[
 P(\epsilon_1)\times P(\epsilon_2) \to P.
 \]
 This map is an isogeny. Indeed, by \eqref{sum_of_genus} both abelian varieties have the same dimension $g-1=(g(C_1)-g(F))+(g(C_2)-g(F))$. Moreover, arguing as in \cite[Lemma~4.4]{N}, we get that $\pi_1^*(JC_1) \cap \pi_2^*(JC_2)=(\epsilon \circ \pi)^*(JF)$, therefore $\pi_1^*P(\epsilon_1) \cap \pi_2^* P(\epsilon_2)$ is finite. Hence, the map has finite kernel and is an isogeny. 

The involution $\sigma $ commutes with $\iota $ and therefore it acts on both curves $C_1$ and $C_2$, and therefore on the Jacobians $JC_1, JC_2$ and their Prym varieties $P(\epsilon_1), P(\epsilon_2)$.
The isogeny given above is equivariant with respect to these actions, being the identity on $P(\epsilon_2)$.
Hence, $\sigma$ is a pseudoreflection on $P$ if, and only if, $\dim P(\epsilon_1)=1$.
This is equivalent to the condition $g(C_1)=g(F)+1$.
The Riemann Hurwitz formula implies the numerical restrictions of the statement.
\end{proof}

In what follows we will assume that $\sigma$ does lift to a pseudoreflection on $P$, and in particular $g(F)\in\{0,1,2\}$.
We will proceed to analyze the situation based on these three cases.
 
\subsection{The case \texorpdfstring{$g(F)=0$}{g(F)=0}}\label{g(F)=0}
This case has already been considered by Mumford \cite{M}, who proves that $P(\epsilon_1)=JC_1$, $P(\epsilon_2)=JC_2$, and $P\cong JC_1\times JC_2$.
Here $JC_1$ is an elliptic curve and $\sigma$ acts on $P$ as $-1$ on $JC_1$ and as the identity on $JC_2$.
In this case, we therefore have that
\[P/\sigma=\mathbb{P}^1\times JC_2\]
is the trivial $\mathbb{P}^1$-bundle over $JC_2$.
We note moreover that this situation never occurs in the case of quotients of Jacobian varieties, since these are irreducible.

The family we obtain in this way is $\cA_1 \times \mathcal {JH}_{g-2}\subset \cA_{g-1}$, where $\mathcal {JH}_{g-2}$ stands for the locus of hyperelliptic Jacobians of dimension $g-2$.
This family is irreducible of dimension $1+2(g-2)-1=2g-4$.

\begin{rem}\label{products}
We note that this is the only case where the Prym variety is a polarized product.
Indeed, if the theta divisor $\Xi$ has two or more components, then the singular locus $\operatorname{Sing}(\Xi)$ has codimension $2$ in $P$.
According to \cite[Theorem~4.10]{be}, this only happens if $C$ is hyperelliptic (we are assuming that $C$ is smooth, otherwise other possibilities appear).
\end{rem}

\subsection{The case \texorpdfstring{$g(F)=1$}{g(F)=1}}

The Prym map for bielliptic curves is considered in \cite{N}, 
where the general fibers are described for $g\ge 10$.
We introduce some notation: let $\cB_g\subset \cM_g$ denote the moduli space of curves with a bielliptic structure; 
that is, smooth curves $C$ of genus $g$ equipped with a degree $2$ map $C\rightarrow E$ over an elliptic curve.
It is known that $\cB_g$ is irreducible of dimension $2g-2$.
The preimage of $\cB_g$ by the forgetful map $\cR_g\rightarrow \cM_g$ breaks-up into several irreducible components, one of which, $\mathcal {RB}_{g,2}$, 
corresponds to the elements $(C,\eta)$ for which there exists a diagram as in \eqref{diagram} with $g(F)=1, g(C_1)=2$ 
(see \cite[Section~2]{N} for more details, where this component is called $\cR_{B,g,1}$). Our aim is to compute the dimension of $\cP_g(\mathcal {RB}_{g,2})$. To do it, and for later applications, it will be useful to know that generically the abelian variety $P(\epsilon_2)$ is simple.
This is a particular case of the following fact, that may have independent interest and, at the same time, probably well-known by the experts.
\begin{proposition}\label{Prym_simple} 
For $g\geq 2$, the general element of the image of the ramified Prym map
\[
\cP^{}_{g,r}: \cR^{}_{g,r} \to \cA^{\delta }_{g-1+\frac r2}
\]
is simple.
\end{proposition}

\begin{proof}
Observe that the infinitesimal deformations of an element $(F,L,B)\in \cR_{g,r}$ are the same as the infinitesimal deformations of $(F,B)$ (since if $B$ is fixed, then $L$ has finitely many options).
The space of these infinitesimal deformations is $H^1(F,T_F(-B))$, the tangent space to the moduli space of genus $g$ curves with $r$ marked points.
The tangent space to the Prym variety $P\in \cP_{g,r}(\cR_{g,r})$ is $T_0P=H^0(F,\omega_F\otimes L)^*$ and the differential of the Prym map:
\[
H^1(F,T_F(-B)) \to \operatorname{Sym}^2 H^0(F,\omega_F\otimes L)^*
\]
is given by the cup-product:
\[
H^1(F,T_F(-B))\otimes H^0(F,\omega_F \otimes L)\to H^1(F,L^{-1})\cong H^0(F,\omega_F\otimes L)^*.
\]
Since there is a natural polarization on $P$, there is a canonical isomorphism $H^0(F,\omega_F\otimes L)^*\cong \overline{H^0( F,\omega_F\otimes L)}$, and composing with the conjugation map we have that the infinitesimal deformations act on $H^0(F,\omega_F \otimes L)$.
Assume that a general $P$ has a proper subabelian variety.
Thus, by a standard argument, there exists a proper subspace $(0)\neq W \subsetneq H^0(F,\omega_F \otimes L)$ such that $\epsilon \cdot W \subset W$, for all $\epsilon \in H^1(F,T_F (-B))$.
Note that $W$ is the tangent space at the origin of the quotient of $P$ by the proper abelian subvariety.

Let $s \in H^0(F,\omega_F \otimes L)$ be a section, and consider the map:
\[
m_{s}:H^1(F,T_F (-B)) \to H^0(F,\omega_F \otimes L)^*, \quad \epsilon \mapsto \epsilon \cdot s.
\]
The dual of this map is multiplication by $s$:
\[
H^0(F,\omega_F \otimes L) \to H^0(F,\omega_F ^2(B)) 
\]
which is obviously injective.
Hence, $m_{s}$ is surjective and, by the invariance of $W$ by deformations, all the sections of $H^0(F,\omega_F \otimes L)$ belong to $W$, a contradiction.
\end{proof}

The main result of this section is the following:

\begin{theorem}\label{RBg2}
Assume $g\ge 5$. Then the bielliptic Prym locus $\cP_g(\mathcal {RB}_{g,2})$ gives an irreducible $(2g-3)$-dimensional family of Prym varieties admitting a pseudoreflection involution.
\end{theorem} 
\begin{proof}
In \cite[Theorem~8.7]{N} it is proved, under the hypothesis $g\ge 10$, that the general fiber of the Prym map restricted to $\mathcal {RB}_{g,2}$ is $1$-dimensional and a complete description of the fiber is given.
This result is optimal since the curve $C$ has a $1$-dimensional family of tetragonal series $g^1_4$ and Donagi's tetragonal construction (see \cite{do_fibres}) implies that the general fiber is at least $1$-dimensional.
In particular, the statement of the theorem follows for $g\ge 10$.
The proof in \cite{N} is based on the analysis of the singularities of the theta divisor of $P$ combined with some ``translation invariant" properties of these loci.
Our aim is to prove that the general fiber has dimension $1$ for $g\ge 5$, maybe with several components.

Let $P=P(\tC,C)$ be an element in $\cP_g(\mathcal {RB}_{g,2})$, where $\tC, C, C_1, C_2,\ldots$ are as in diagram \eqref{diagram}, and let $\pi ' \colon \widetilde D \rightarrow D$ be another double covering in $\mathcal {RB}_{g,2}$ such that $P=P(\widetilde D,D)$. 
We use similar notation: $D_1, D_2, F', \epsilon', \epsilon_1 ', \epsilon_2' ,\ldots$ for the corresponding diagram.
On $P$ we have two pseudoreflections coming from the two representations of $P$ as the Prym variety of $(\tC,C)$ and $(\widetilde D,D)$.
By \autoref{Prym_simple}, we can assume that $P(\epsilon_2)$ is simple and therefore is the only codimension $1$ abelian subvariety of $P$. Therefore $P(\epsilon_2)=P(\epsilon_2')$. 
Observe that the maps $\epsilon_2, \epsilon_2'$ are double coverings of an elliptic curve ramifying in $2g-4$ points, thus they belong to the image of the ramified Prym map $\cP_{1, 2g-4}$.
Since we assume that $g\ge 5$, we have that $2g-4\ge 6$ and by the Torelli Theorem for ramified Prym maps (see \cite{NO}) the coverings $\epsilon_2$ and $\epsilon_2'$ are isomorphic.
In particular, $F=F'$, and up to a finite number of choices (the number of automorphisms of $C_2$ is finite) we can assume that $\epsilon _2=\epsilon_2'$.

Moreover, $P(\epsilon_1)$ is determined by the embedding $P(\epsilon_2)\subset P$. Indeed, it is the so-called polarized complementary (see \cite[Section~ 5.3]{bl}) of $P(\epsilon_2)$ in $P$. Hence $P(\epsilon_1)=P(\epsilon_1')$.

Therefore, we have to consider the fiber at $P(\epsilon_1)$ of the ramified Prym map $\cP_{1,2}:\cR_{1,2}\rightarrow \cA_1$ restricted to the set: 
\[\cY_F=\set{(F,\eta, x_1+x_2) \mid \eta \in \Pic^1(F), x_1+x_2 \in \vert \eta^2 \vert, x_1 \neq x_2 },\]
where $F$ is our fixed elliptic curve.
An easy computation with tangent spaces shows that $\cP_{1,2}$ restricted to $\cY_F$ is a dominant map of curves, hence, up to a finite number of choices we can assume that $\epsilon _1=\epsilon_1'$.

The pair of coverings $(\epsilon_1,\epsilon_2)$ does not determine completely the pair $(\tC,C)$, since $\epsilon_1$ can be composed with an automorphism of $F$ before doing the fiber product over $F$.
In other words, there is a $1$-dimensional family of coverings in $\mathcal {RB}_{g,2}$ that share the coverings $(\epsilon_1,\epsilon_2)$. As explained in detail in \cite[Subsection~(2.10)]{N} this is exactly the effect of the tetragonal construction and implies that the fiber is $1$-dimensional.
\end{proof}

\subsection{The case \texorpdfstring{$g(F)=2$}{g(F)=2}}\label{sec:g=2}
Using \eqref{sum_of_genus} and assuming $g(F)=2$ and $g(C_1)=3$, we obtain that $g(C_2)=g(C)=g$ and the branch divisor $B$ of $\epsilon$ has degree $2g-6$.
Note that by Riemann Hurwitz formula in this case $\pi_2$ is \'etale.
It is easy to completely describe the diagram starting from some data given on $F$.
Indeed, assume that the map $\epsilon$ is given by the data $(F,L,B)$, where $L\in \Pic^{g-3}(F)$ and $B$, the branch divisor of $\epsilon$, is a reduced divisor in the linear system $\vert L^2 \vert$.
Let $\eta_1\in JF[2]$ be the 2-torsion point in $JF$ defining the étale map $\epsilon_1$.
Then $(F,L,B,\eta_1)$ determines and is determined by the diagram \eqref{diagram} as a result of the following lemma.

\begin{lemma}\label{lemma_g=2}
 With the above notation, the following holds:
 \begin{enumerate}
 \item\label{item:a:g=2} We have $\epsilon^* \eta_1\cong \eta$.
 \item\label{item:b:g=2} The map $\epsilon_2$ is given by the data $(F, L\otimes \eta_1,B)$.
 In particular, the branch divisor of $\epsilon$ and $\epsilon_2$ is the same.
 \item\label{item:c:g=2} The curve $\tC$ is the fiber product of $\epsilon$ and $\epsilon_1 $ (respectively $\epsilon_1$ and $\epsilon_2$).
 \item\label{item:d:g=2} The map $\iota $ defines a pseudoreflection on the Prym variety $P(\pi_2)$ and $\sigma $ defines a pseudoreflection on $P(\pi)$.
 \end{enumerate}
\end{lemma}
\begin{proof}
To prove \eqref{item:a:g=2} observe that 
\[\cO_{\tC}=\pi_1^*\epsilon_1 ^*\eta_1=\pi^* \epsilon ^* \eta_1 .\]
Hence $\epsilon^* \eta_1$ belongs to the kernel of $\pi^*$, which is generated by $\eta$. We only have to exclude that $\epsilon^* \eta_1 $ is trivial.
This follows from the injectivity of $\epsilon ^*$.

The branch divisor of $\pi \circ \epsilon =\pi_2\circ \epsilon_2$ is $B$. Since $\pi_2 $ is \'etale, also the branch divisor of $\epsilon_2$ is $B$. Hence, $\epsilon_2$ is determined by the data $(F,M,B)$, where $M$ is a line bundle on $F$ such that $B\in \vert M^2\vert$. In particular $M^2\cong L^2\cong \cO_F(B)$, thus $M\cong L\otimes \alpha $ for some $2$-torsion point in $JF$. To finish the proof of \eqref{item:b:g=2} we have to show that $\alpha = \eta_1 $. Note that $\alpha $ can not be trivial, otherwise $C_2=C, \pi_2 = \pi $, $\iota = \sigma $ and therefore $C_1 = \tC /\iota \sigma =\tC $. This is impossible since $g(C_1)=3$. By the basic properties of double covers we have:
$$
\omega_{\tC}=\pi_2^* \omega_{C_2} =\pi_2^* \epsilon_2^* L\otimes \alpha =
\pi^* \epsilon^* L\otimes \alpha = \omega_{\tC}\otimes \pi^* \epsilon ^* \alpha.
$$
Thus $\epsilon^* \alpha =\eta=\epsilon^* \eta_1 $ and then $\alpha=\eta_1$.

Note that \eqref{item:c:g=2} is a consequence of general properties of the fiber product and \eqref{item:d:g=2} follows from the previous discussions, since $\dim \pi_2^*P(\epsilon_2)=\dim P(\pi)-1$ and $\dim \pi^*P(\epsilon)=\dim P(\pi_2)-1$. 
\end{proof}
Now let $\cF_g\subseteq\mathcal{M}_g$ denote the locus of genus $g$ curves that are double covers of a genus 2 curve, branched over $2g-6$ points, let $\cR\cF_{g}$ denote its preimage in $\cR_g$ via the forgetful map, and let $\cR\cF_{g,3}$ denote the locus of all elements in $\cR\cF_g$ such that the quotients of $\tC$ by each lifting of the automorphism $\sigma$ are of genus $3$ and $g$, respectively.

\begin{lemma}\label{lem:Fg} Assume $g\geq 5$.
Then we have the following:
\begin{enumerate}
\item $\cF_g$ is irreducible of dimension $2g-3$, and the automorphism group of a general element of $\cF_g$ is $\bZ/2\bZ$.
\item $\cR\cF_{g,3}$ is irreducible.
\end{enumerate}
\end{lemma}
\begin{proof}
The first statement follows from the analysis done after Remark 1 in Page 138 of \cite{Cornalba}, as well as from Theorem~1 of the same article.

To prove the second statement, we use the notation and results of \autoref{lemma_g=2}.
In particular, it is enough to prove that $\cT:=\{(F,L,B,\eta_1)\}/_{\cong}$ is irreducible, since $\cT$ dominates $\cR\cF_{g,3}$.
We consider the forgetful map $h:\cT \rightarrow \cR_2$ sending $(F,L,B,\eta_1)$ to $(F,\eta_1)$.
Note that, for any $(F,\eta_1)$, the fiber $h^{-1}(F,\eta_1)$ is isomorphic to the fiber of the forgetful map $\cR_{2,2g-6}\rightarrow \cM_2$.
Let us see that this fiber, that we denote by $H_F$, is irreducible of constant dimension.
By definition, we have:
\[
H_F=\set{(L,B)\in \Pic^{g-3}(F)\times (F^{(2g-6)}\setminus \Delta) \mid B\in \vert L^2 \vert },
\]
where $\Delta $ stands for the main diagonal.
The projection map $H_F\rightarrow \Pic^{g-3}(F)$ is surjective since $g\ge 5$ and $g(F)=2$.
Moreover, for any $L\in \Pic^{g-3}(F)$, $h^0(F,L^2)=h^0(F,\omega_F\otimes L^{-2})+2g-6+1-2=2g-7$.
Therefore, $H_F$ has a map on a Jacobian of dimension $2$ whose fibers are open sets in a projective space of dimension $2g-8$.
This finishes the proof of the statement.
\end{proof}

\begin{theorem}\label{RFg3}
Assume $g\geq 5$.
Then the Prym locus $\cP_g(\cR\cF_{g,3})$ gives a $(2g-3)$-dimensional family of Prym varieties with a pseudoreflection.
Moreover, the restriction of the Prym map $\cP_g|_{\cR\cF_{g,3}}$ has finite fibers for $g\geq6$, and is generically finite for $g=5$.
\end{theorem}

\begin{proof}
For $g\geq 6$ we need to prove that $\cP_g|_{\cR\cF_{g,3}}$ is quasifinite.
For this, let $(P,\Xi)\in \cP_g(\cR\cF_{g,3})$.
Since $\Aut(P,\Xi)$ is a finite group, in particular there are finitely many pseudoreflections.
If we take one pseudoreflection $\sigma\in\Aut(A,\Xi)$, then we want to show that there are finitely many pairs $(C,\eta)\in\cR\cF_{g,3}$ with an involution that induce $\sigma$.

If $(C,\eta)\in\cR\cF_{g,3}$ is a preimage of $(P,\Xi)$, then define $Z:=\Fix_P(\sigma)_0=\pi_2^*P(\epsilon_2)$ where $\epsilon_2:\tC/\sigma\to C/\sigma$ is the natural map as before. The polarization on $P$ induces on $Z$ (twice) the same polarization of the polarization induced by 
that of $\pi_2^*(JC_2)$ which is of type $(1,\ldots,1,2)$. 
Let $\widetilde {\mathcal A}:=\widetilde {\mathcal {A}}_{g-1}^{(1,\ldots,1,2,2)}$ be the moduli space of triples $(A,L,\alpha)$ with $\alpha $ a $2$-torsion point of $A$ in the kernel of $\varphi_L:A\to A^{\vee }$. The following map
\[
\widetilde {\mathcal A}\longrightarrow 
\mathcal {A}_{g-1}^{(1,\ldots,1,2)}, 
\]
sending $(A,L,\alpha)$ to the quotient $A/\alpha $ with its natural polarization. Then, up to a finite number of choices we recover $P(\epsilon_2)$.

Now by Riemann Hurwitz, $\iota$ acts on $\tC/\sigma$ with $2g-6\geq 4$ fixed points.
If $g\geq6$, by \cite[Theorem~1.1]{NO}, $\epsilon_2$ is the unique ramified double cover whose Prym variety is equal to $(Z,\Xi|_{Z})$.
Therefore, given $(P,\Xi,\sigma)$, we can recover $\epsilon_2$.
On the other hand, the \'etale cover $\tC/\iota\sigma\to C/\sigma$ is given by a non-trivial 2-torsion point $\eta_1\in J(C/\sigma)[2]$, of which there are finitely many.
This implies that $\tC=(\tC/\iota\sigma)\times_{C/\sigma}\tC/\sigma$, has finitely many possibilities once we fix the pseudoreflection $\sigma$.
Moreover, we see that once we define $\tC$, $\iota$ is immediately defined as the lifting of the involution associated to the \'etale double cover $\tC/\iota\sigma\to C/\sigma$.
We conclude that $\cP_g|_{\cR\cF_{g,3}}$ is quasifinite.

For $g=5$, the same analysis holds, but the ramified Prym map \[\cP_{5,4}:\cR_{5,4}\to\cA_{6}^{(1,2,2,2,2,2)}\]
is only generically injective, and by \cite{NO}, it actually has fibers of positive dimension.
\end{proof}

\begin{corollary}
The degree of the general fiber of $\cP_g|_{\cR\cF_{g,3}}$ is at most $15\cdot 2^{2g-4}$.
\end{corollary}
\begin{proof}
We have already seen at the beginning of Subsection~\ref{sec:g=2} that an element of the image of $\cP_g|_{\cR\cF_{g,3}}$ depends on a tuple $(F,L,B,\eta_1)$.
For a general $F$, $L$ and $B$, we have that the Prym variety associated to the double cover is simple by \autoref{Prym_simple}.
This implies that for a general $(P,\Xi)$ in the image of $\cP_g|_{\cR\cF_{g,3}}$, we have an isogeny $P\simeq E\times Z$, where $E$ is an elliptic curve and $Z$ is simple. Note that this implies that $E$ is the only elliptic curve lying on $P$, since any other elliptic curve $F$ would induce a non-zero homomorphism $(f,g):F\to E\times Z$. Since the image of $g$ would be isogenous to $F$ if $g\neq 0$, then by the simplicity of $Z$, we have that necessarily $g=0$ and so $F=E$. Given that $P$ is in the image of $\cP_g|_{\cR\cF_{g,3}}$, it necessarily contains a pseudoreflection. On the other hand, since a pseudoreflection on an abelian variety of dimension $g-1$ must have $-1$ as an eigenvalue with multiplicity $1$ and $1$ as an eigenvalue with multiplicity $g-2$, and since the only codimension $1$ abelian subvariety of $P$ is $Z$, the only possible pseudoreflection in this case is given by multiplication by $-1$ on $E$ and the identity on $Z$. This implies that the given pseudoreflection is unique.
Therefore, the different possibilities for constructing $P$ from a curve $(C,\eta)$ depend  on the choice of $\eta_1$ among the nontrivial $2$-torsion points of $JF$ and on the coverings of degree $2$ of $Z$ as in the proof of Theorem \ref{RFg3}. Hence, we obtain the bound in the statement.
\end{proof}

\subsection{A distinguished example with \texorpdfstring{$g(F)=2$}{g(F)=2} (Cubic threefolds with Eckardt points).}\label{Eckardt}
Prym varieties of dimension $5$ with a pseudoreflection appear naturally when considering intermediate Jacobians of cubic threefolds with Eckardt points, that is, points such that there are infinitely many lines in the cubic threefold passing through them.
We refer to \cite{CZ}, with special attention to Section~7 as a reference for this subsection.

Let $X\subset \mathbb P^4$ be a cubic threefold and let $p\in X$ an Eckardt point.
Then, there is an involution $\sigma_p $ of $\mathbb P^4$ such that $\sigma_p(X)=X$.
The set of fixed points by $\sigma_p $ consists of $p$ and a hyperplane $H$ with $p\notin H$.
We denote by $S$ the cubic surface $H\cap X$ which is smooth.
Recall that for any line $l$ in the Fano surface, $F(X)=\set{r\in \operatorname{Gr}(2,5)\mid r\subset X}$, the projection to a supplementary plane of $l$ provides a conic-bundle structure on $X$, that is, a map from the blow-up of $X$ at $l$ to $\bP^2$ whose fibers are conics:
\[
\pi_l:\operatorname{Bl}_l X \to \bP^2.
\]
The discriminant plane curve, where the fibers of $\pi_l$ degenerate, is a plane quintic.
There is a curve $\widetilde Q_l\subset F(X)$ parametrizing the lines contained in these degenerate conics. Observe that $\widetilde Q_l$ can be seen as the set of lines intersecting $l$. In particular, $\sigma_p$ acts on $\widetilde Q_l$.

It turns out that the Prym variety of the natural covering $\widetilde Q_l\rightarrow Q_l$ is isomorphic, as principally polarized abelian varieties, to the intermediate Jacobian $JX$.
As it is proved in \cite[Lemma~7.10]{CZ}, the quintic given by any of the $27$ lines $l\in F(X)$ contained in the cubic surface $S$, comes equipped with a degree $2$ map $f:Q_l\rightarrow F$ on a curve $F$ of genus $2$.
Moreover, the involution on $Q_l$ determined by $f$ induces on $P(\widetilde Q_l,Q_l)\cong JX$ the same pseudoreflection as the involution $\sigma_p$ on $X$.
Note that the action of $\sigma_p$ on $\widetilde Q_l$ has exactly $12$ fixed points ($2$ correspond to the lines in $X\cap(p\vee l)$ and $10$ to the lines contained in the cubic surface $X\cap H$ and intersecting $l$). Then, by Riemann-Hurwitz formula $g(C_1)=g(\widetilde Q_l/\langle \sigma_p \rangle)=3.$  
Hence $(\widetilde Q_l,Q_l)\in \cR\cF_{6,3} $ and $JX$ is an example of Prym variety with a pseudoreflection in $\cP_g(\cR\cF_{6,3})$ with $27$ representations as Prym variety of the double covering of a quintic plane curve.

Note that this example also shows that there are Prym varieties with more than one pseudoreflection of geometric origin.
Indeed, consider a cubic threefold $X$ with equation:
\[
z \, Q(u,v)+ (x^2+y^2)\, L(z,u,v)=0,
\]
where $L, Q$ are appropriate forms of degrees $1$ and $2$ respectively such that the threefold is smooth.
Then $[1:0:0:0:0]$ is an Eckardt point with associated involution $[x:y:z:u:v]\mapsto [-x:y:z:u:v]$.
In the same way $[0:1:0:0:0]$ is another Eckartd point with involution $[x:y:z:u:v]\mapsto [x:-y:z:u:v]$.
The hyperplane of fixed points for the first involution is $x=0$ and for the second $y=0$.
The line $l:=\{x=y=z=0\}$ belongs to both cubic surfaces $X\cap \{x=0\}$ and $X\cap \{y=0\}$.
Hence the discriminant quintic of the conic bundle given by the projection from $l$ has two natural involutions and, according with the description above, the intermediate Jacobian $JX$ has two different pseudoreflections of geometric origin.

It is worth mentioning that \cite{bo26} contains examples of $5$-dimensional Prym varieties with pseudoreflections of geometric origin that are not intermediate Jacobians of cubic threefolds.

\subsection{The three families are different}\label{different}
We have found, assuming $g\ge 5$, the existence of three irreducible families of Prym varieties with an action of a finite group of geometric origin and with smooth quotient: 
\[
\cA_1 \times \mathcal {JH}_{g-2}, \cP_g(\mathcal {RB}_{g,2}), \cP_g(\mathcal {RF}_{g,3}) \subset \cA_{g-1}.
\]
The first has dimension $2g-4$, and the other two have dimension $2g-3$.
By the discussion at the beginning of Section~\ref{sec:involutions}, any Prym variety with such an action belongs to one of these families.
It is natural to ask if they are different and, more generally, to compute how big the intersection of two of them can be. By the Castelnuovo-Severi inequality (see e.g. \cite[Theorem~3.5]{Ac})
 bielliptic curves of genus at least $4$ are never hyperelliptic.
Combined with \autoref{products}, we deduce that:
\[
\cA_1 \times \mathcal {JH}_{g-2} \cap \cP_g(\mathcal {RB}_{g,2}) = \emptyset.
\]
To understand the intersection of $\cA_1 \times \mathcal {JH}_{g-2}$ with $ \cP_g(\mathcal {RF}_{g,3}) $ we need to know when a curve of genus $g\ge 5$ with a degree $2$ map on a curve of genus $2$ is hyperelliptic.

\begin{lemma}\label{lem:hyper}
Let $\epsilon : C\rightarrow F$ be a degree two map of smooth curves with $g(C)\ge 5$ and $g(F)=2$.
Let $L\in \Pic ^{g-3}(F)$, $B\in \vert L^2 \vert$ the data determining $\epsilon $.
We have:
\begin{enumerate}
 \item\label{casea:g6} If $g\ge 6$, then $C$ is not hyperelliptic.
 \item\label{caseb:g5} If $g=5$ and $C$ is hyperelliptic, then $L=\omega_F$.
\end{enumerate}
\end{lemma}

\begin{proof}
Part \eqref{casea:g6} is again a consequence of the Castelnuovo-Severi inequality quoted above.
Part \eqref{caseb:g5} can be found in \cite[Proposition~3.2]{bo}.
\end{proof}

\begin{corollary}
 For $g\ge 6$, we have $\cA_1 \times \mathcal {JH}_{g-2} \cap \cP_g(\mathcal {RF}_{g,3}) = \emptyset$.
 For $g=5$, we have 
 \[
 \dim \left(\cA_1 \times \mathcal {JH}_{3} \cap \cP_5(\mathcal {RF}_{5,3})\right)\le 5.
 \]
 \end{corollary}
\begin{proof}
By \autoref{lem:hyper}, the first statement follows and only the statement for $g=5$ has to be proved.
Since $L=\omega_F$ (the $g^1_2$ of $F$), the divisor $B\in \vert L^2\vert$ moves in a projective space of dimension $2$ and $F\in \cM_2$.
Thus, the intersection depends on, at most, $5$ parameters.
\end{proof}

Note that $\dim \cA_1 \times \mathcal {JH}_{3}=6$, hence, even in genus $5$, the families are different.
Finally, we prove the following:

\begin{proposition}\label{prop:different}
The families $\cP_g(\mathcal {RB}_{g,2})$ and $\cP_g(\mathcal {RF}_{g,3})$ have different closure.
\end{proposition}
\begin{proof}
We proceed by contradiction, assume that the two subvarieties of $\cA_{g-1}$ have the same closure.
Let $(P,\Xi)$ be a general Prym variety in this family.
A crucial observation is that $(P,\Xi)$ is not a Jacobian variety.
Indeed, this is a parameter count: the Prym variety $P$ contains an elliptic curve and the subvariety
\[
\cJ^1_{g-1}=\{(JN,\Theta _N)\in \cJ_{g-1} \mid JN \text{ contains an elliptic curve }\}\subset \cA_{g-1}
\]
can be identified with the locus of curves $N\in\cM_{g-1}$ with a morphism onto an elliptic curve.
This is a countable union of subvarieties of dimension $2(g-1)-2=2g-4$.
Since $\dim \cP_g(\mathcal {RB}_{g,2})=2g-3$, we are done.

If $(P,\Xi)\in \cP_g(\mathcal {RB}_{g,2})$ is general, then $\dim \operatorname{Sing} \Xi =g-5$.
This can be found in \cite[Proposition~5.2.2]{De} 
for $g\ge 6$.
For $g=5$ we observe the following: let $(\tC,C)\in \mathcal {RB}_{5,2}$. The map $\epsilon \colon C\to E$ is determined by a  branch divisor and a line bundle $L\in \Pic^4(E)$. Note that $\omega_C=\epsilon^*L$. 
Let $M_i$, $1\le i\le 4$, be the four line bundles such that $M_i^2=L$.
Then the pull-backs $\pi^*\epsilon^*M_i$ to $\tC$ have norm $\omega_C$ and $4$ global sections.
According to \cite{M} these provide singularities of the theta divisor $\Xi$.
Hence $\operatorname{Sing} \Xi  \neq \emptyset$. Moreover, if $\dim \operatorname{Sing} \Xi >0$, then by  
the classification provided by \cite[Theorem~4.10]{be}, $P$ would be a hyperelliptic Jacobian, a contradiction. Thus, for all $g\ge 5$ we have 
$\dim \operatorname{Sing} \Xi =g-5$. Going back to the classification in \cite[Theorem~4.10]{be}, we obtain that:
\begin{enumerate}
\item\label{casea:trigonal} either $C$ is trigonal,
\item\label{caseb:bielliptic} or $C$ is bielliptic of genus $g\ge 6$,
\item\label{casec:nullvierte} or $C$ is a genus $5$ curve with a theta-characteristic $N$ with $2$ sections and $h^0(\tilde C,\pi^*N)$ is even,
\item\label{cased:planequintic} or $C$ is a genus $6$ smooth plane quintic and $h^0(\tilde C, \pi^*\cO_C(1))$ is odd.
\end{enumerate}
We can exclude the case \eqref{casea:trigonal} in our discussion, since by Recillas' construction (see \cite[\S2.4]{do_fibres}) the Prym variety of an unramified double cover of a trigonal curve is always a Jacobian variety.

Now we represent our general Prym variety in the family as the image of an element $(\tC,C)$ in $\mathcal {RF}_{g,3}$.
Observe that,  by \autoref{RFg3}, the map $\cP_g|_{\cR\cF_{g,3}}$ has finite fibers, hence we can assume that $C$ is general. Now we prove that a general curve $C$ of genus $g$ with a degree $2$ map onto a curve of genus $2$ is not bielliptic. 
Indeed, if $C$ were general with a map $C\rightarrow F$, $g(F)=2$, then its Jacobian would be isogenous to $P(C,F)\times JF$ and $P(C,F), JF$ are simple (by \autoref{Prym_simple}), so $C$ is not bielliptic.

We consider now the case $g=6$.
We want to see that a general curve of genus $6$ with a degree $2$ map $\epsilon : C\rightarrow F$ onto a curve of genus $2$ is not a plane quintic.
Assume that $C$ is a plane quintic with an involution $f$ with exactly six fixed points.
This is equivalent to the existence of a map $C\rightarrow F$ with $g(F)=2$.
The involution leaves $\cO_C(1)$ invariant since $W^2_5(C)$ consists of exactly one point 
and therefore $f$ is the restriction of a linear involution of the projective plane, that we still denote by $f$, leaving $C$ invariant.
Thus, there is a line $r\subset \bP^2$ of pointwise fixed points for $f$ and at most another fixed point not belonging to $r$.
Then $r$ intersects $C$ in $5$ different points and there is a fixed point $P\in C\setminus r$.

Considering coordinates such that $P=[1:0:0]$ and $r=\{x=0\}$ we have that $f([x:y:z])=[x:-y:-z]$.
Since $C$ is invariant by $f$, its equation $H$ satisfies that $H(x,-y,-z)=\lambda H(x,y,z)$.
One easily sees that $\lambda =\pm 1$.
The case $\lambda =1$ implies that $x$ divides $H$, hence we assume that $\lambda = -1$.
In this case $H$ is of the form
\begin{equation}\label{eq_quintic}
\sum_{i+j+k=5} A_{ijk}x^i y^j z^k=0,\quad \quad A_{ijk}=0 \quad \text{ if } \quad j+k \quad \text{ is even}.
\end{equation}
In particular $A_{500}=0$.
All these equations corresponding to smooth quintics are para\-metrized by an open set of $\mathbb P^{11}$.
Then, up to the action of the elements of $\operatorname{PGL}(3)$ preserving the line $x=0$ and the point $[1:0:0]$ (which is $4$-dimensional), the family is $7$-dimensional.
Since the locus of curves of genus $6$ with a map onto a curve of genus $2$ has dimension $9$, we are done.

The genus $5$ case is more involved. The following lemma finishes the genus $5$ case and therefore the proof of the \autoref{prop:different}.
\end{proof}

\begin{lemma}
Let $C$ be a curve of genus $5$ with a degree $2$ morphism $\epsilon : C\rightarrow F$ onto a curve $F$ of genus $2$. Then $C$ does not have a semicanonical pencil (a theta-characteristic with two sections).
\end{lemma}
\begin{proof}
Let us denote by $\cF$ the closure in $\cM_5$ of the locus of curves $C$ with a degree $2$ morphism onto a curve $F$ of genus $2$, and $\cT\subset \cM_5$ the irreducible divisor of curves with a semicanonical pencil, see \cite{Te}. Note that $\dim \cF=7$. Our aim is to prove that $\cF \not\subset  \cT$.

We first note that $C$ is neither hyperelliptic nor trigonal. This is an easy consequence of the inequality in \cite[Lemma 3.1]{N}.
Hence, $C$ can be embedded in $\mathbb P^4$ as the complete intersection of $3$ quadrics.
Note that $f$ induces an involution on $\mathbb P^4$ that we still denote by $f$.
The quadrics containing $C$ determine a plane, and the degenerate quadrics of the plane give a plane quintic. It is well-known that the quadrics of rank $3$ containing $C$ are in bijection with the semicanonical pencils in $C$ and with the nodes of $Q_C$ (see \cite[Cap VI, Exer. F]{ACGH}). This quintic $Q_C$ comes equipped with a natural unramified double covering $\widetilde {Q_C}$ provided by the rulings of the quadrics and $P(\widetilde {Q_C},Q_C)=JC$.
In the literature, the closure of the locus of coverings as above in $\widetilde {\cR_6}$ (the Beauville partial compactification of $\cR_6$, see \cite{be}) is denoted by $\mathcal{RQ}^+$. The Prym map gives a birational map between $\mathcal{RQ}^+$ and the Jacobian locus $\cJ_5$. To be more precise, following \cite[Subsection~(4.3)]{do_fibres}, the description given above gives a map
\[
\cJ_5 \dashrightarrow \mathcal{RQ}^+\subset \widetilde {\cR_6} 
\]
which is the inverse of the Prym map restricted to $\mathcal{RQ}^+$. We assume that $\cF\subset \cT$. By what we said before, by applying $\cT$ to $\mathcal{RQ}^+$ and projecting to $\overline {\cM_6}$, we obtain nodal quintics.
Therefore, the quintics coming from curves in $\cF$ would have also a node.
As we noticed, the quintics coming from curves of $\cF$ are invariant by an involution of the plane.
We gave in the case $g=6$ a parametrization of the quintics with an involution (that has to be the restriction of an involution of the plane), we obtained an irreducible subvariety of $\overline{\cM_6}$ of dimension $7$.
This proves that the image of $\cF$ is exactly this family of quintics invariant by involutions.
Since a general polynomial of the form  \eqref{eq_quintic}, with $A_{500}=0$, corresponds to a non-singular curve, we reach a contradiction.  
\end{proof}

\section{List of potential groups acting on Prym varieties with smooth quotient}\label{listofgroups}

We recall from \cite[Theorem~1.1]{ALA} that if $(P,\Xi)$ is a principally polarized abelian variety and $H\leq\operatorname{Aut}(P,\Xi)$ is such that $P/H$ is smooth, then there is an isogeny $P\simeq X\times Y$ where $H$ acts trivially on $X$, and on $Y$ with finitely many fixed points.
Moreover, there exist elliptic curves $E_1,\ldots,E_r,F_1,\ldots,F_s$ such that
\[Y\cong \prod_{i=1}^rE_i^{n_i}\times\prod_{j=1}^sF_j^{m_j}\]
 \begin{equation}\label{possibleG} H= \left[\prod_{i=1}^r(\bZ/m_i\bZ)^{n_i}\rtimes S_{n_i}\right]\times\prod_{j=1}^sS_{t_j+1}\end{equation}
 where $m_i\in\{2,3,4,6\}$, $n_i\geq2$, $t_j\geq1$, and each factor of $H$ acts on the corresponding factor of $Y$ in an explicit way. The following result severely restricts the possible groups of geometric origin that can act on a Prym variety:

 \begin{proposition}
If $(C,\eta)$ corresponds to the double covering $\tC\to C$ with Prym variety $(P,\Xi)$, and $G\leq\operatorname{Aut}(\tC)$ is a subgroup such that each of its elements commutes with $\iota$, let $\rho:G\to\operatorname{Aut}(P,\Xi)$ denote the natural representation of $G$ on $P$. If the genus of $C$ is greater than or equal to $5$, then we have the following: 

\begin{enumerate}
\item If $P/\rho(G)$ is smooth, then $G$ is a $2$-group. 
\item If $\ker\rho$ is trivial, then $G$ must be isomorphic to one of the following nine groups:
\[(\bZ/2\bZ)^i\times(\bZ/4\bZ)^j\hspace{0.5cm}\text{for }0\leq i,j\leq 2,i+j\leq 2\]\[(\bZ/2\bZ)^3,\hspace{0.3cm}(\bZ/2\bZ)^2\rtimes S_2,\hspace{0.3cm}(\bZ/4\bZ)^2\rtimes S_2,\hspace{0.3cm}\left[(\bZ/2\bZ)^2\rtimes S_2\right]\times S_2.\]
\item If $\ker\rho$ is trivial and $\Xi$ is irreducible, then the only possibilities for $G$ are $(\bZ/2\bZ)^i$ for $i=1,2,3$.

\end{enumerate}
\end{proposition}

\begin{proof}
Write $H=\rho(G)$ as in equation \eqref{possibleG}. First, by \autoref{order3}, we have that if $G$ possesses an element of order $3$, then the codimension of its fixed point locus in $P$ must be at least 3. However, we see from equation \eqref{possibleG} and by the action that each factor has, that if either $n_i\geq 3$ or $t_j\geq2$, then there exists a 3-cycle that fixes a codimension 2 locus. Indeed, any two 3-cycles are conjugate in the symmetric group, and we can therefore take a look at the 3-cycle $(1\,2\,3)$. If we are in the case $n_i\geq 3$ or $t_j\geq 3$, then this element just acts on a self-product of an elliptic curve $E^n$ by permuting the first three coordinates. We see in this case that the fixed locus of the 3-cycle is then exactly $\{(x,x,x,x_4,\ldots,x_n):x,x_4,\ldots,x_n\in E\}$, which is of codimension $2$. In the case $t_j=2$, then we are in the case of $E^2$, and the element of order $3$ then acts as $(x,y)\mapsto(-x-y,x)$. We see that its fixed locus is precisely the subvariety $\{(x,x):x\in E[3]\}$, which is finite and therefore of codimension $2$.

Therefore, we automatically have that $m_i\in\{2,4\}$, $n_i\in\{1,2\}$, and $t_j=1$, and so we can write
\[\rho(G)=(\bZ/4\bZ)^a\times\left[\prod_{i=1}^b(\bZ/m_i\bZ)^{2}\rtimes S_{2}\right]\times\prod_{j=1}^sS_{2}.\]
By \autoref{lem:ord2kernel}, we have that the order of $\ker\rho$ is at most $2$, and therefore $G$ must be a $2$-group.

For the proof of the following two items, we will assume that $\rho$ is injective, and we can therefore identify $G$ with $H=\rho(G)$.

If $\sigma\in G$ is a pseudoreflection of order $2$, then we have that the subgroup $C_G(\sigma)$ of elements in $G$ that commute with $\sigma$ descends to an action on $\tilde{C}/\iota\sigma$. We will analyze the possibilities for $G$ case by case, according to the three cases studied in the preceding sections.

If $(C,\eta)\in\cA_1\times\mathcal{JH}_{g-2}$ and $s\neq 0$, we can take a pseudoreflection $\sigma$ from one of the factors $S_2$, and obtain that $G/\sigma$ acts on $JC_2$ with smooth quotient.
Since the action also fixes the polarization, we get a Jacobian with a smooth quotient.
By \cite[Theorem~1.2]{ALA}, we have that $g(C_2)\leq 3$ and $G/\sigma$ must be isomorphic to either $\bZ/2\bZ$, $(\bZ/2\bZ)^2$ or $S_3$ (the theorem that is cited actually states that only $\bZ/2\bZ$ can appear, but in order to discard the other two groups, the proof uses \cite[Lemma~4.1]{ALA} which was found to contain a mistake. Therefore, the proof of \cite[Theorem~1.2]{ALA} only shows that $G$ is either isomorphic to $\bZ/2\bZ$, $(\bZ/2\bZ)^2$ or $S_3$).
Now this implies that necessarily $a=b=0$, and therefore $s\leq 3$.
On the other hand, if $s=0$ and $b\neq 0$, then we can take $\sigma$ in one of the $(\bZ/2\bZ)^2\rtimes S_2$ as one of the elements that multiplies a coordinate by $-1$;
the same analysis gives us that if we take another pseudoreflection that commutes with $\sigma$, then $JC_2$ has a smooth quotient by this pseudoreflection and therefore $g(C_2)\leq 3$ by \cite[Theorem~1.2]{ALA}.
In particular, we also obtain that $a=0$, $b=1$, $m_1=2$, and so $G=(\bZ/2\bZ)^2\rtimes S_2$. Finally, if $b=s=0$ and $a\neq 0$, then a similar analysis leads us to conclude that $a$ must be equal to $1$.

If $(C,\eta)\in \mathcal{P}_g(\mathcal{RB}_{g,2})$, then the genus of $\tilde{C}/\iota\sigma$ is $2$.
If $b\geq 1$, let $\sigma$ be a pseudoreflection of order $2$ in one of the copies of $\bZ/m_1\bZ$.
Then $C_G(\sigma)=(
\bZ/4\bZ)^a\times(\bZ/m_1\bZ)^{2}\times\left[\prod_{i=1}^{b-1}(\bZ/2\bZ)^{2}\rtimes S_{2}\right]\times\prod_{j=1}^sS_{2}$.
By looking at Broughton's list \cite{Broughton}, the only group that acts in genus 2 of this form is $(\bZ/2\bZ)^2$.
Therefore, here $b=1$ and $a=s=0$, and so $G=(\bZ/2\bZ)^{2}\rtimes S_{2}$.
If $b=0$, then $G$ is commutative and so must act on $\tilde{C}/\iota\sigma$.
For the same reason as before, we have that either $a=1$ and $s=0$, or $a=0$ and $s\in\{1,2\}$. 

If $(C,\eta)\in\mathcal{P}_g(\mathcal{RF}_{g,3})$, we have that the genus of $\tilde{C}/\iota\sigma$ is $3$.
If $b\geq 1$, again, looking at Broughton's list, we see that $b=1$, $a=0$, and $s\in\{0,1\}$. This gives us the following possibilities for $G$:
\[(\bZ/2\bZ)^2\rtimes S_2,(\bZ/4\bZ)^2\rtimes S_2,\left((\bZ/2\bZ)^2\rtimes S_2)\right)\times\mathbb{Z}_2.\]
If $b=0$, then $G$ is abelian, and again by Broughton's list, we obtain the following possibilities for $G$:
\[(\bZ/2\bZ)^i\times(\bZ/4\bZ)^j\hspace{0.5cm}\text{for }0\leq i,j\leq 2,i+j\leq 2.\]
In the irreducible case, the main result of \cite{ALA}, together with the previous work, says that $G$ must be a product of cyclic groups of order $2$, and the statement follows.
\end{proof}

\begin{rem}
It is not clear if all of these groups actually act on Prym varieties (with geometric origin) with a smooth quotient, and it would indeed be an interesting challenge to see if examples of these groups acting on Prym varieties with smooth quotients can be constructed.
In this article we showed that the groups $\bZ/2\bZ$ and $(\bZ/2\bZ)^2$ can appear (see the subsection on \hyperref[Eckardt]{cubic threefolds with Eckardt points} for the latter case).
\end{rem}


\begin{thebibliography}{ACGH85}

\bibitem[Acc94]{Ac}
Robert D.~M. Accola.
\newblock {\em Topics in the theory of {R}iemann surfaces}, volume 1595 of {\em
  Lecture Notes in Mathematics}.
\newblock Springer-Verlag, Berlin, 1994.

\bibitem[ACGH85]{ACGH}
Enrico Arbarello, Maurizio Cornalba, Phillip~A. Griffiths, and Joseph Harris.
\newblock {\em Geometry of algebraic curves. {V}ol. {I}}, volume 267 of {\em
  Grundlehren der mathematischen Wissenschaften [Fundamental Principles of
  Mathematical Sciences]}.
\newblock Springer-Verlag, New York, 1985.

\bibitem[ALA20]{ALA1}
Robert Auffarth and Giancarlo Lucchini~Arteche.
\newblock Smooth quotients of {A}belian varieties by finite groups.
\newblock {\em Ann. Sc. Norm. Super. Pisa Cl. Sci. (5)}, 21:673--694, 2020.

\bibitem[ALA22]{ALA}
Robert Auffarth and Giancarlo Lucchini~Arteche.
\newblock Smooth quotients of principally polarized abelian varieties.
\newblock {\em Mosc. Math. J.}, 22(2):225--237, 2022.

\bibitem[Bea77]{be}
Arnaud Beauville.
\newblock Prym varieties and the {S}chottky problem.
\newblock {\em Invent. Math.}, 41(2):149--196, 1977.

\bibitem[BL04]{bl}
Christina Birkenhake and Herbert Lange.
\newblock {\em Complex abelian varieties}, volume 302 of {\em Grundlehren der
  mathematischen Wissenschaften [Fundamental Principles of Mathematical
  Sciences]}.
\newblock Springer-Verlag, Berlin, second edition, 2004.

\bibitem[BO24]{bo}
Pawe\l~ Bor\'owka and Angela Ortega.
\newblock Involutions on hyperelliptic curves and prym maps.
\newblock {\em Ann. Sc. Norm. Super. Pisa Cl. Sci. (5)}, 2024.

\bibitem[BO26]{bo26}
Pawe\l~ Bor\'owka and Angela Ortega.
\newblock Klein coverings over hyperelliptic genus 3 curves, 2026.
\newblock arXiv:2602.18969.

\bibitem[Bro91]{Broughton}
S.~Allen Broughton.
\newblock Classifying finite group actions on surfaces of low genus.
\newblock {\em J. Pure Appl. Algebra}, 69(3):233--270, 1991.

\bibitem[Bud24]{Bu}
Andrei Bud.
\newblock The birational geometry of {$\overline{\mathcal R}_{g,2}$} and
  {P}rym-canonical divisorial strata.
\newblock {\em Selecta Math. (N.S.)}, 30(2):Paper No. 30, 31, 2024.

\bibitem[CMZ21]{CZ}
Sebastian Casalaina-Martin and Zheng Zhang.
\newblock The moduli space of cubic surface pairs via the intermediate
  {J}acobians of {E}ckardt cubic threefolds.
\newblock {\em J. Lond. Math. Soc. (2)}, 104(1):1--34, 2021.

\bibitem[Cor87]{Cornalba}
Maurizio Cornalba.
\newblock On the locus of curves with automorphisms.
\newblock {\em Ann. Mat. Pura Appl. (4)}, 149:135--151, 1987.

\bibitem[Deb88]{De}
Olivier Debarre.
\newblock Sur les vari\'et\'es ab\'eliennes dont le diviseur theta est
  singulier en codimension {$3$}.
\newblock {\em Duke Math. J.}, 57(1):221--273, 1988.

\bibitem[Don92]{do_fibres}
Ron~Donagi.
\newblock The fibers of the {P}rym map.
\newblock In {\em Curves, {J}acobians, and abelian varieties ({A}mherst, {MA},
  1990)}, volume 136 of {\em Contemp. Math.}, pages 55--125. Amer. Math. Soc.,
  Providence, RI, 1992.

\bibitem[Ike20]{Ik}
Atsushi Ikeda.
\newblock Global {P}rym-{T}orelli theorem for double coverings of elliptic
  curves.
\newblock {\em Algebr. Geom.}, 7(5):544--560, 2020.

\bibitem[Mum74]{M}
David~Mumford.
\newblock Prym varieties. {I}.
\newblock In {\em Contributions to analysis (a collection of papers dedicated
  to {L}ipman {B}ers)}, pages 325--350. Academic Press, New York, 1974.

\bibitem[Nar92]{N}
Juan-Carlos Naranjo.
\newblock Prym varieties of bi-elliptic curves.
\newblock {\em J. Reine Angew. Math.}, 424:47--106, 1992.

\bibitem[NO22]{NO}
Juan~Carlos Naranjo and Angela Ortega.
\newblock Global {P}rym-{T}orelli for double coverings ramified in at least six
  points.
\newblock {\em J. Algebraic Geom.}, 31(2):387--396, 2022.

\bibitem[TiB88]{Te}
Montserrat Teixidor~i Bigas.
\newblock The divisor of curves with a vanishing theta-null.
\newblock {\em Compositio Math.}, 66(1):15--22, 1988.

\end{thebibliography}
\end{document}